\documentclass[a4paper,12pt,english]{smfart}
\usepackage[utf8]{inputenc}
\usepackage{graphicx}

\usepackage[T1]{fontenc}
\usepackage[french]{babel}

\DeclareUnicodeCharacter{200B}{{\hskip 0pt}}

\usepackage{fullpage}

\usepackage[arrow, matrix, curve]{xy}
\usepackage{latexsym}
\usepackage{amsfonts}
\usepackage{amsmath,mathrsfs,amssymb,amsthm,amscd}
\usepackage{url}
\usepackage{hyperref}

\usepackage[table]{xcolor}

\newcommand{\Q}{\mathbb{Q}}
\newcommand{\F}{\mathbb{F}}
\newcommand{\Z}{\mathbb{Z}}

\def\a{{\mathfrak a}}
\def\b{{\mathfrak b}}

\def\p{{\mathfrak p}}
\def\q{{\mathfrak q}}
\def\l{{\mathfrak l}}
\def\qq{{\mathfrak Q}}

\def\K{{\tilde K}}

\def\FF{{\tilde F}}

\def\O{{\mathcal O}}
\def\I{{\mathcal I}}
\def\U{{\mathcal U}}
\def\J{{\mathcal J}}

	\newtheorem{thm}{Theorem}[section]
	\newtheorem{lem}[thm]{Lemma}
	\newtheorem{propo}[thm]{Proposition}
	\newtheorem{cor}[thm]{Corollary}

\def\Ell{\mathfrak L}

\newtheorem{Theorem}{Theorem}
	\newtheorem{defn}[thm]{Definition}
	\newtheorem{rem}[thm]{Remark}

\newtheorem{Corollary}[Theorem]{Corollary}

\newtheorem{Remarque}{Remark}

\def\H{{\mathcal C}}

\title{On  $S$-split $p$-Hilbert class field towers with prescribed Galois groups}

\date{\today}

\author{Christian Maire}
 \address{Université Marie et Louis Pasteur,  CNRS, Institut FEMTO-ST, F-25000 Besançon, France}
\email{christian.maire@univ-fcomte.fr}

\author{Karim Sankara}
 \address{Nazi Boni University, Bobo-Dioulasso, Burkina Faso \and Université Marie et Louis Pasteur,  CNRS, Institut FEMTO-ST, F-25000 Besançon, France}
\email{karim.sankara@femto-st.fr}

\date{\today}


\thanks{The authors are very grateful to  Ravi Ramakrishna for his interest in our work and helpful comments. They also thank  the International Mathematical Union, Commission for Developing Countries (IMU CDC) and the GRAID program for their support. This work has been supported by the EIPHI Graduate School (contract "ANR-17-EURE-0002") and by the Bourgogne-Franche-Comté Region. The authors would also like to thank the anonymous referee for several helpful suggestions that improved the article.}

\begin{document}
	
\maketitle

\begin{center}
\emph{Dedicated to the memory of our friend and colleague  Tony Ezome (1979--2024).}
\end{center}

\

\

\begin{abstract}
In this work, we show that given a finite $p$-group $G$, a number field $K$ having a trivial $p$-class group $Cl_K$, and a finite set of primes $S$ of $K$, there exists a finite extension $F/K$ such that the  $S$-split $p$-Hilbert class field tower $L_p^S(F)$ of $F$ has $G$ as its Galois group.
 This extends results by Ozaki and Hajir-Maire-Ramakrishna.
\end{abstract}

\section*{Introduction}

Let $p$ be a prime number and $K$ be a number field. Denote by $L_p(K)$ the top of the $p$-Hilbert class field tower of $K$, which is the maximal unramified $p$-extension of $K$. The extension $L_p(K)/K$ can also be constructed by iteratively stacking the $p$-Hilbert class fields $K^{(i)}$: here, $K^{(0)} = K$, and $K^{(i+1)}$ is the $p$-Hilbert class field  of $K^{(i)}$, {\it i.e.}, the maximal unramified abelian $p$-extension of $K^{(i)}$. Recall that the Artin map provides an isomorphism between the $p$-Sylow subgroup $Cl_{K^{(i)}}$ of the class group of $K^{(i)}$ and $Gal(K^{(i+1)}/K^{(i)})$. Let~$G_K := Gal(L_p(K)/K)$.

Observe that $L_p(K) = K$ if and only if the $p$-part $Cl_K$ of the class group of $K$ is trivial (which is the case, for example, when $K = \Q$). On the other hand, the Golod–Shafarevich criterion shows that the $p$-extension $L_p(K)/K$ can be infinite (see \cite{GS},  \cite[\S 7.7]{Koch} or \cite{R}).
To the best of our knowledge, using this criterion is the only method available to identify infinite $p$-towers. This naturally leads to the following question:

{\it Does every finite $p$-group $G$ arise as the Galois group of the $p$-Hilbert class field tower of some number field $K$?}

This question can be viewed as an inverse Galois problem for the $p$-Hilbert tower.

Ozaki answered this question affirmatively in \cite{O}. This result was revisited and further extended in \cite{HMR}. It is this version that forms the basis of our approach.

\medskip

Here, we focus on the inverse Galois problem for the $p$-Hilbert class field towers with decomposition. Let us clarify the context.

Let $S$ be a set of primes of $K$.
Denote by $L_p^S(K)$ the maximal $p$-extension of $K$ that is unramified everywhere in which every prime in $S$ splits completely. Let $G_K^S := Gal(L_p^S(K)/K)$. The extension $L_p^S(K)/K$ is also the largest Galois subextension of $L_p(K)/K$ fixed by the decomposition groups of the primes above  $S$.  In particular, $G_K^S$ is a quotient of $G_K$.

The groups $G_K^S$ have been the subject of extensive study; see, for example, \cite[Chapter 11]{Koch}, \cite[Chapter III]{Gras}, \cite[Chapter X]{NSW}, \cite{M}, etc.

\medskip

Before stating the main result of this note, let us introduce some notations.

\smallskip

For a finite $p$-group $G$, let $h^i_G = \dim H^i(G, \mathbb{Z}/p)$, and let $p^{e_G}$ denote the exponent of~$G$.
For a number field $K$, denote its signature by $(r_{K,1}, r_{K,2})$.

\begin{Theorem} \label{main_theorem} Let $K$  be a number field with a finite $p$-Hilbert tower $L_p(K)/K$; set $G:=G_K =Gal(L_p(K)/K)$.
Assume that  $r_{K,1}+r_{K,2}\geq h^1_G + h^{2}_G$.

Let  $S$ be a finite set of primes of  $K$.

Then there exists a tamely ramified   extension $F/K$ of degree $p^m$ such that
\begin{itemize} \item[$(i)$]  $L_{p}(F) = L^S_{p}(F)$;
 \item[$(ii)$]  the Galois group $Gal(L^{S}_{p}(F)/F)$ is isomorphic to  $G$;
 \item[$(iii)$] the extension $F/K$ is ramified at  $m$ primes;
 \item[$(iv)$] $m\leq e_G$.
\end{itemize}
\end{Theorem}

Here, we make a slight abuse of notation by still denoting by $S$ the set $S_F$​ of primes of~$F$ lying above the primes in the original set $S$.

In Theorem \ref{main_theorem}, the proof shows that $\# S_F=\#S$.

\begin{Remarque} The choice of $F$ depends on $S$. However,
we can observe that the estimates on the degree of $F/K$ and on the number of primes ramified in $F/K$ do not depend on $S$.
\end{Remarque}

The  technical condition $r_{K,1}+r_{K,2}\geq h^1_G + h^{2}_G$ ensures the presence of a sufficient number of Minkowski units associated with $K$ (see \S \ref{section_MU} for the definition).
When $\zeta_p \notin K$, the condition on the signature of $K$ can be refined (see Theorem \ref{main_theorem_0}).
This technical condition is satisfied in the main result of \cite{HMR}, allowing us to deduce the following corollary.
\begin{Corollary} \label{coroB}
Let $K$ be a number field with a trivial $p$-class group $Cl_K$​, and let~$S$ be a finite set of primes of $K$. Let $G$ be a finite $p$-group. Then there exists an extension $F/K$, tamely ramified and unramified at infinity, such that the Galois group $Gal(L^S_p(F)/F)$ is isomorphic to $G$.
\end{Corollary}

\begin{Remarque} The previous results establish the existence of an extension $F/K$ with very specific properties. This extension is constructed by successively applying Chebotarev's theorem, which in fact ensures the existence of infinitely many extensions satisfying the mentioned properties.
\end{Remarque}

The proof of the main result relies critically on the presence of Minkowski units along the $p$-Hilbert tower $L_p(K)/K$. We then need to {\it eliminate} the residue degree at $S$ within a given $p$-tower. This is achieved by forming the compositum with extensions  $\tilde{F}/\tilde{K}$ that are inert at $S$, while maintaining the stability of the $p$-tower, a point where we use a result found in \cite{HMR}. The existence of the extensions  $\tilde{F}/\tilde{K}$ is ensured by carefully selecting Frobenius elements in appropriate governing fields.

\medskip

Our work is organized into three parts. In \S 1, we recall the necessary tools. In \S 2, we prove Theorem \ref{main_theorem}. Finally, in \S 3, we conclude with three remarks: one concerns  a certain condition $\H_S$, another concerns the degrees of the fields encountered, and the third addresses the same problem in the context of $p$-Hilbert towers with tame ramification and decomposition.

\

{\bf Preliminaries}

We fix an algebraic closure $\overline{\Q}$ of $\Q$. Let $p$ be a prime number. The element $\zeta_p \in\overline{\Q}$​  denotes a primitive $p$th root of unity.

\smallskip

$\bullet$ Let    $K$ be a number field.
\begin{itemize}
\item[$-$] A prime $\q$ of the ring of integers $\O_K$ of $K$ is called {\it tame} if $\# \O_K/\q \equiv 1  \ (mod \ p)$.
\item[$-$] A $\Z/p$-extension $F/K$ refers to a cyclic extension of degree $p$.
\item[$-$] The field $K(\zeta_p)$ is denoted by $K'$.
\end{itemize}

\smallskip

 $\bullet$  Let $S = \{\l_{1},...,\l_{s}\}$ be a finite set of  primes of  $K$.
 \begin{itemize}
\item[$-$]
 $\O_{K}^S = \{\alpha \in K\mid  v_{\p}(\alpha)\geq 0, \ \forall \p\not \in S \}$ is the ring of  $S$-integers of $K$,  where  $v_\p$ is the  $\p$-valuation (normalized) associated to the prime $\p$ of $K$.
 \item[$-$] The group $E^S:=E_{K}^S$ of $S$-units of $K$ is the group of invertible elements of  $\O_{K}^S$.
 \item[$-$]
 We recall that $E^S$ is isomorphic to  $W_K\times \Z^r  $, where $W_{K}$ is the cyclic group of the roots of unity in $K$, and  $r = r_{K,1} + r_{K,2} - 1 + s$.
 \end{itemize}

 \smallskip

$\bullet$  The $p$-Sylow subgroup of the  $S$-class group of $K$  denoted by $Cl^{S}_{K}$ is defined
 by $$Cl^{S}_{K} := \Z_p\otimes \left(I_{K}/\mathcal{P}_{K}\langle S\rangle \right),$$
 where $I_K$ is the group of fractional ideals of $K$, $\mathcal{P}_{K}$ is the subgroup of principal fractional ideals, and $\langle S \rangle$ is the subgroup of  $I_K$ generated by the prime ideals in~$S$. It is also canonically isomorphic to the class group of $\O_K^S$.
 Recall that, by class field theory, $Cl_K^S$ is canonically isomorphic to the Galois group of the maximal abelian $p$-extension $K^{S,(1)}/K$, unramified everywhere and completely split at $S$ (real archimedean places remain real).
Set $Cl_K^S[p]:=\{h \in Cl_K^S, h^p=1\}$.

\smallskip

$\bullet$ For a prime $\p$ of $K$, let $K_\p$ be the completion of $K$ at $\p$, and  $U_\p$ be the subgroup of units of $K_\p^\times$.
Let  $\J:=\J_K$ be the group of the idèles of $K$; set $U_\infty=\prod_{v|\infty} K_v^\times$,  where the products is taken over archimedean places $v$ of $K$; set $\U:=\U_K=U_\infty \prod_\p U_\p $.

Observe that $Cl_K \simeq \Z_p\otimes \left(\J/K^\times \U\right)$, and $Cl_K^S\simeq \Z_p\otimes \left(\J/K^\times \U \prod_{\l\in S} K_\l^\times \right)$.

\smallskip

$\bullet$ More generally,
let $S$ and $T$ be two finite and disjoint sets of primes of $K$. We assume that the primes in $T$ are tame.
Let $L_{p,T}^S$ denote the maximal pro-$p$ extension of $K$ which is unramified outside $T$ and in which all primes in $S$ split completely.
Observe that $L_{p,\emptyset}^S=L_p^S$ and $L_{p,\emptyset}^\emptyset=L_p(K)$. Set $G_{K,T}^S:=Gal(L_{p,T}^S/K)$.

 \smallskip

 $\bullet$
 If $G$ is a $p$-group, we denote by $\Phi(G):=G^p[G,G]$ its Frattini subgroup.
 All cohomology groups have coefficients $\Z/p$ (with trivial action) so we write  $H^i(G)$ for $H^i(G,\Z/p)$. We denote by $h^i_G$ the dimension over $\F_p$ of $H^i(G)$. Recall that $h^1_G$ is the minimal number of generators of $G$, and $h^2_G$ is the minimal number of relations.

 \smallskip

 $\bullet$ For $p=2$, the set of extensions considered in this paper are unramified at infinity.

 %%%%%%%%%%%%%%%%%%%%%%%%%%%%%%%%%%%%%%%%%%%ù
 %%%%%%%%%%%%%%%%%%%%%%%%%%%%%%%%%%%%%%%%%%%%%%%%%%%

\section{Tools for the proof}

\subsection{Governing field}

\subsubsection{Definition}

Given a finite set~$S$ of primes of $K$,  we define the multiplicative subgroup $V_{K}^S$ of~$K^{\times}$~as
$$ V_{K}^S = \{ x\in K^{\times}, (x) \in  (I_K)^{p}\langle S\rangle\}.$$
When $S=\emptyset$, we denote $V_K:=V_K^{\emptyset}$.

\begin{lem} \label{lemmeVS} The following exact sequence holds:
$$1 \longrightarrow E_{K}^S/(E_{K}^S)^p \longrightarrow V_{K}^S/(K^\times)^p \longrightarrow Cl_K^S[p] \longrightarrow 1 \cdot$$
\end{lem}

\begin{proof} The map $f:V_K^S/(K^\times)^p \rightarrow Cl_K^S[p]$ is defined as follows. Take $x\in V_K^S$: there exists $\a\in I_K$ and $\b \in \langle S\rangle$ such that $(x)=\a^p \b$. Then, $f(x)$ is the class of $\a$ in $Cl_K^S$. The map $f$ is surjective, and the kernel is exactly $E_K^S/(E_K^S)^p$. \end{proof}

\begin{defn} The governing field relatively to $K$ and $S$ is the number field  $Gov_{K}^S:=K'(\sqrt[p]{V_{K}^S})$. When $S=\emptyset$, we denote $Gov_K:=Gov_K^\emptyset$.
\end{defn}

 The extension  $Gov_{K}^S/K$ is Galois. Set  $M_K^S=Gal(Gov_K^S/K')$:  it is an elementary abelian $p$-group of $p$-rank which can be deduced  from Lemma \ref{lemmeVS}. Again, when $S=\emptyset$, we denote  $M_K:=M_K^\emptyset$.

 \smallskip

 For the groups $M_K^S$, we will use additive notation.

 \subsubsection{Governing field and idèles} \label{section_ideles}

Let $S=\{\l_1,\cdots, \l_s\}$ and $T=\{\q_1,\cdots, \q_t\}$ be two finite and disjoint sets of primes of $K$. We assume that the primes $\q_i$ in $T$ are tame.
Set

  \begin{itemize}
  \item[$\bullet$]  $\U_T^S:= U_\infty \prod_{\l \in S} K_\l^\times \prod_{\p \notin S\cup T} U_\p $, $\U^S:=\U_\emptyset^S$, $\U_T:=\U_T^\emptyset$,
  \item[$\bullet$] $V_{T}^S:=V_{K,T}^S=\{x \in K^\times \mid x\in U_\p, \ \forall \p \notin S\cup T; x\in (K_\q^\times)^p, \ \forall \q \in T\}$.
 \end{itemize}

One can express the sets $V_{T}^S$ in terms of idèles:
$$V_T^S=K^\times \cap \J^p \U_T^S\cdot$$ In particular, $V^S=K^\times \cap  \J^p \U^S$  and $V_T=K^\times \cap \J^p \U_T$.

 \subsubsection{Governing field and Frobenius} \label{section_governing_frobenius}

 Let us choose a tame prime $\q$ of $K$, not in $S$. Since $\# \O_K/\q \equiv 1 \ (mod \ p)$, the prime~$\q$   splits completely in  $K'/K$. Let $\qq$ be a prime ideal of   $K'$ above $\q$. Denote by  $\sigma_\qq$ the Frobenius of  $\qq$ in
$ Gal(Gov_K^S/K')$. For convenience, we will write  $\sigma_\q:=\sigma_\qq$ for some specific  $\qq|\q$. Note that if  $\qq'$ is another prime of  $K'$ above~$\q$, then  $\sigma_{\qq} = a  \sigma_{\qq'}$ for some  $a \in \F_p^\times$.   It follows that any property involving  $\sigma_\q$ will not depend on the choice of  $\qq|\q$.

Let $S$ and $T$ be as before.
Set $Gov_{K,T}^S:=K'(\sqrt[p]{V_{K,T}^S})$, and $M_{K,T}^S=Gal(Gov_{K,T}^S/K')$.

\begin{lem} \label{lemm_generate_frob}
 The Galois group $Gal(Gov_K^S/Gov_{K,T}^S)$ is generated by the Frobenius elements $\sigma_\q$ at $\q \in T$.
\end{lem}

\begin{proof}
 First observe that primes $\q\in T$ are unramified in $Gov_K^S/K'$.

 Take  $x\in V_{K,T}^S$, and $\q\in T$. Then $x\in (K_\q^\times)^p$,  $(Gov_{K,T}^S)_\q=K_\q$, and consequently $Gov_{K,T}^S$ is fixed by $\sigma_\q \in M_K^S$.

 Reciprocally. Take $x \in V_K^S$ such that for all $\q\in T$, $\sigma_\q \in Gal(Gov_K^S/ K'(\sqrt[p]{x}))$. Then $x\in (K_\q^\times)^p$ (recall that $\zeta_p \in K_\q$). Hence, $x \in V_{K,T}^S$.
\end{proof}

\begin{rem}
 Observe that if a prime $\p$ of $K$, coprime to $p$ and $S$, has a non-trivial Frobenius in $M_K^S$ then $\p$ is tame.
\end{rem}

 \subsubsection{Governing field and $\Z/p$-extensions with prescribed ramification}

Let us state the following theorem due to Gras (see \cite[Chapter V, Corollary 2.4.2]{Gras}), which is not a priori useful in the proof of our result but which we will use in a remark in  \S \ref{section_remark_HS}.

Let $S$ and $T$ be as before.

\begin{thm}[Gras] \label{theorem_gras}  There exists a $\Z/p$-extension  $F/K$ that is exactly ramified at~$T=\{\q_1,\cdots, \q_t\}$ and totally decomposed at~$S$ if and only if, for $i = 1,...,t$,  there exists $a_{i}\in \F_p^\times$
such that:
$$
 \sum^{t}_{i = 1}a_i \sigma_{\q_i} = 0 \in M_K^S.
$$
\end{thm}

Here ``exactly ramified'' at~$T$ means that $F/K$ is unramified outside  $T$ and every prime in   $T$ is ramified in $F/K$.
We will use the following corollary.

\begin{cor} \label{coro_gras} Assume $S=\{\l\}$ and let  $\q$ be a tame prime  of $K$, $\q \neq \l$,
such that $\sigma_\q \neq 0 $ in $M_K^S$. If there exists a   $\Z/p$-extension $F/K$ that is exactly ramified at  $\q$, then  $\l$ is inert in $F/K$.
\end{cor}

%%%%%%%%%%%%%%%%%%%%%%%%%%%%%%%%%%

\subsection{Governing fields in the $p$-Hilbert tower}\label{section_treillis}

  Our main result relies on the choice of Frobenius elements in $M_{L_p(K)}$ ​ while considering the action of $G_K$​.
  To do this, we will need the following properties of linear disjunction.

\begin{lem}\label{prop_disj_lin} Let $L/K'$ be an unramified extension. Then  $$L\cap K'(\sqrt[p]{V_{K}^S})=L\cap K'(\sqrt[p]{V_{K}})$$ and $$ L(\sqrt[p]{V_L}) \cap L(\sqrt[p]{V_K^S})= L(\sqrt[p]{V_K}).$$
\end{lem}

\begin{proof}
For the first point, obviously $L\cap K'(\sqrt[p]{V_{K}})\subset L\cap K'(\sqrt[p]{V_{K}^S})$.
Now, let $x \in V_{K}^S$​ such that $K'(\sqrt[p]{x})/K'$ is unramified. This implies that $x\O_{K'} \in \I_{K'}^p$​, and we will see that  $x\in V_K$​. Indeed, it is sufficient to take the norm in $K'/K$  and then observe that  $([K':K],p)=1$. Hence, the first point is proved.

Now let  $x \in V_{K}^S$ such that  $L(\sqrt[p]{x}) \subset L(\sqrt[p]{V_L})$. By Kummer theory,  there exists $z\in V_L$ such that $z x^{-1} \in (L^\times)^p$. Therefore, $x \O_L \in \I_L^p$.
On the other hand, in~${K}$, we have  $(x)=\a^p \b$, with $\b \in \langle S\rangle$.
Since the extension $L/K'$ is unramified at  $S$ (and $[K':K]$ is prime to $p$), we deduce that $\b \in \I_K^p$, which implies $x \in V_K$. Thus, we have $$ L(\sqrt[p]{V_L}) \cap L(\sqrt[p]{V_{K}^S})=L(\sqrt[p]{V_L})\cap L(\sqrt[p]{V_K})=L(\sqrt[p]{V_K}),$$ since $V_K \subset V_L$.
\end{proof}

Consider now  the following extensions.

$$
\xymatrix{&& F_2\ar@{-}[r] & F_5 \\
L_{p}(K) \ar@{-}[r] & L_p(K)(\zeta_p)\ar@{-}[r] \ar@{.}@/^1.5pc/[ur]^{M_{L_p(K)}}  & F_1 \ar@{-}[r] \ar@{-}[u]  & F_4  \ar@{-}[u] \\
& \ar@{-}[r]& F_0 \ar@{-}[u] \ar@{-}[r]&F_3\ar@{-}[u] \\
K\ar@{-}[uu]  \ar@{.}@/^1pc/[uu]^{G_K}  \ar@{-}[r]& K(\zeta_p) \ar@{-}[uu] \ar@{.}@/_1pc/[urr]_{M_K^S}  \ar@{.}@/_1pc/[ur]^{M_K} \ar@{.}@/^1pc/[uu]^{G_K}   &  &
}
$$

 where  $F_{0} := K^{\prime}(\sqrt[p]{V_K})=Gov_K$, $F_1:=  L_{p}(K) F_0$,
 $F_2 :=  L_{p}(K) \left(\zeta_p,\sqrt[p]{E_{L_{p}(K)}}\right)=Gov_{L_p(K)}$,   $F_3:= K^{\prime}\left(\sqrt[p]{V_{K}^S}\right)=Gov_K^S$, $F_4=F_3F_1=L_p(K)F_3$, and $F_5=F_3F_2$.

 Observe that $V_L=E_{L_p(K)}$.

 \smallskip

 Since  $L_p(K)(\zeta_p)/K'$  is unramified, by Lemma \ref{prop_disj_lin}, we have   $F_1 \cap F_3=  F_0$ and $F_2 \cap F_4=F_1$. Therefore, $F_2\cap F_3 = F_2 \cap F_4 \cap F_3= F_1 \cap F_3= F_0$.

 \smallskip

 In conclusion,  the  extensions $F_2/F_0$ and $F_3/F_0$ are linearly disjoint over $F_0$, and we thus have
 \begin{eqnarray}\label{iso_galois} Gal(F_5/F_{0}) = Gal(F_5/F_2)\times Gal(F_5/F_3) \simeq Gal(F_3/F_0) \times Gal(F_2/F_0).\end{eqnarray}

\subsection{Minkowski units} \label{section_MU}

\subsubsection{Structure} \label{section_MU_structure}
First, let us recall a well-known result.

\begin{propo} Let $G$ be a $p$-group, and let  $M$  be a finitely generated  $\F_p[G]$-module. Then every free submodule of  $M$ is a direct summand of $M$. In particular, there exists a well-defined integer  $\lambda \geq 0$ such that $M\simeq \F_p[G]^{\lambda}\oplus N$, where  $N$ is a torsion $\F_{p}[G]$-module. %De plus, on a $$t\geq d_p M  - (\vert G\vert - 1)d_p M^{G}.$$
\end{propo}

\begin{proof}
This result follows from the fact that  $\F_p[G]$ is a Frobenius algebra (see \cite[\S 1, (3.15.E)]{La}, and from the   Krull-Schmidt theorem). See also \cite[\S 4]{O}.% C Voir \cite[$\S $ 6 Proposition 6.6]{FM}
\end{proof}

 We apply this result to the following context: $G_K=Gal(L_p(K)/K)$ is finite, and  $M=E_{L_p(K)}/(E_{L_p(K)})^p$.  We then denote $\lambda_K:=\lambda$.

\begin{defn}  A unit $\epsilon \in E_{L_p(K)}$  is called  a Minkowski unit (relative to $L_p(K)/K$) if its class in  $E_{L_p(K)}/(E_{L_p(K)})^p$ generates a free $\F_p[G_K]$-module
\end{defn}

\begin{rem} We actually use  the Kummer dual of $V_{L_p(K)}/(K^\times)^p $, which is itself equal to the Kummer dual of $E_{L_p(K)}/(E_{L_p(K)})^p$, as $p$ does not divide the order of the class group of  $L_p(K)$. Minkowski units of  $L_p(K)/K$ provide a free part of the same rank for this dual, {\it i.e.}, for the group  $M_{L_p(K)}$ viewed as an  $\F_p[G_K]$-module.
\end{rem}

\begin{rem} \label{rem_artin}  Suppose that  $\zeta_p \in K$.  Let  $\q$ be a tame prime of  $K$ that splits completely in  $L_{p}(K)/K$.
There are exactly  $\#G_K$ primes  $\qq_i$ of  $L_{p}(K)$ above $\q$, on which the group  $G_K$ acts transitively. Consider such a prime  $\qq$, and let  $\sigma_\qq $ be its  Frobenius element  in the Galois group  $M_{L_p(K)}$ of the governing field $ Gov_{L_p(K)}$. By the property of the Artin symbol,   $G_K$ acts transitively on  the $\sigma_\qq$.
When $\zeta_p \notin K$, $G_K$ acts transitively on the lines spanned by the
Frobenius elements  $\sigma_\qq \in M_{L_p(K)}$ (see also the observation made in \S \ref{section_governing_frobenius}).
\end{rem}

%%%%%%%%%%%%%%%%%%%%%%%

\subsubsection{Growth of the number of Minkowski units} \label{section_minoration}
We fix a number field $K$ such that $G_K$​ is finite. Let us begin with a definition:

 \begin{defn} Set  $$ A_{K} = \left\{
 \begin{array}{ll}
 r_{K,1} +  r_{K,2} - h^{2}_{G_K} + h^{1}_{G_K} - 1  \quad &\mbox{if}\quad \zeta _{p}\notin  K\\
 r_{K,1} +  r_{K,2} - h^{2}_{G_K} \quad & \mbox{if}\quad \zeta_{p} \in K.
 \end{array}
 \right.
 $$
 \end{defn}

 By \cite[\S 2, Fact 5]{HMR}, one has  $\lambda_K \geq A_K$.

 The presence of Minkowski units is central to our study, and the phenomenon of growth due to a base change  becomes highly significant. This was observed by Ozaki in \cite{O} and quantified in \cite{HMR} through the following proposition:

 \begin{propo}\label{MU1} Let $F/K$ be a  $\Z/p$-extension unramified at infinity and such that   $L_p(F)=FL_p(K)$.    Then, $A_F = A_K + (p-1)\left(r_{K,1} + r_{K,2}\right)$.
 \end{propo}

\begin{proof}
 See {\cite[Proposition 2.6]{HMR}}.
\end{proof}

 %%%%%%%%%%%%%%%%%%%%%%%%%%%%%%%%%%%%%%%%%%%%%%%
 %%%%%%%%%%%%%%%%%%%%%%%%%%%%%%%%%%%%%%%%%%%%%%%%%%%%%%%%%%%%%%%%%%%%%%%%%

\subsection{Stability of the $p$-Hilbert class field tower} \label{section_stabilite}

We begin with a number field~$K$ such that $G_K$​ is finite.

\smallskip

Let $\{g_1,..., g_d \}$ be a minimal system of generators of $G_K = Gal(L_{p}(K)/K)$, where $d =h^1_{G_K}$. The augmentation ideal $I_{G_K}$  of $\F_p[G]$ is generated as a $G$-module by the elements $x_{i} := g_i - 1$,  for $i = 1,...,d$.%, i.e,  $I_{G} = \langle x_{i}\rangle_{i = 1}^{d}$.

\smallskip

Suppose  $\lambda_K \geq d$, and write $M_{L_p(K)}:=Gal(Gov_{L_p(K)}/L_p(K)(\zeta_p))\simeq  \F_p[G_K]^d \oplus M_0$.

In this notation, set  $$z = ((x_{1},...,x_d),0)\in M_{L_p(K)}.$$

Observe that $z\in I_{G_K} \left(M_{L_p(K)}\right)$.
In particular, $z \in Gal\left(Gov_{L_p(K)}/L_p(K)Gov_K\right)$. Indeed, $M_{L_p(K)}/I_{G_K} \left(M_{L_p(K)}\right)$ is the maximal extension of $L_p(K)$ on which $G_K$ acts trivially, and $G_K$  obviously acts trivially on $Gal(L_p(K)Gov_K/L_p(K))$.

\smallskip

The result concerning the stability of the $p$-tower is as follows.

\begin{thm} \label{theorem_stabilite} Suppose that $A_K \geq d$. Let $\q$ be a tame prime of  $K$ such that $\sigma_{\q} = z \in  M_{L_p(K)} \subset
Gal(Gov(L_p(K))/K)$. Then,
there exists a  $\Z/p$-extension $F/K$ exactly ramified at  $\q$ and such that $L_{p}(F) = F L_{p}(K)$.
Moreover $A_F > A_K$.
\end{thm}

\begin{proof} See \cite[Theorem 1]{HMR}.
\end{proof}

%%%%%%%%%%%%%%%%%%%%%%%%%%%%%%%%%%ù
%%%%%%%%%%%%%%%%%%%%%%%%%%%%%%%%%%%%%%%%%%%%%%%%ùù

\section{Main result}

Let us recall the main result of our work (Theorem \ref{main_theorem}).

\begin{thm} \label{main_theorem_0}
Let $K$  be a number field with a finite $p$-Hilbert tower $L_p(K)/K$; set $G:=G_K =Gal(L_p(K)/K)$.
Assume that  $A_K\geq h^1_G$.

Let  $S$ be a finite set of primes of  $K$.

Then there exists a tamely ramified   extension $F/K$ of degree $p^m$ such that
\begin{itemize} \item[$(i)$]  $L_{p}(F) = L^S_{p}(F)$;
 \item[$(ii)$]  the Galois group $Gal(L^{S}_{p}(F)/F)$ is isomorphic to  $G$;
 \item[$(iii)$] the extension $F/K$ is ramified at  $m$ primes;
 \item[$(iv)$] $m\leq e_G$.
\end{itemize}
\end{thm}

Observe that when $G=\{e\}$, the result is immediate. Suppose now $G$ to be nontrivial.
In this case, the condition $r_{K,1}+r_{2,K} \geq h^1_G+h^2_G$ from Theorem \ref{main_theorem} implies that $A_K \geq h^1_G$.

\subsection{Proof of Theorem \ref{main_theorem_0} } \label{proof}
Let $\Sigma=\{ \l_1,\cdots \}$ and $T=\{\q_1,\cdots \}$ be two finite and disjoint sets  of  primes of $K$. We assume that the primes $\q_i \in T$ are tame.
We use the notation from  \S \ref{section_ideles}.

\begin{lem}\label{lem_longue_es}
 We have the exact sequence:
 $$
  V_T/(K^\times)^p \hookrightarrow V_T^\Sigma/(K^\times)^p \longrightarrow \prod_{\l \in \Sigma} K_\l^\times/(K_\l^\times)^pU_\l \longrightarrow \J/ \J^pK^\times  \U_T  \twoheadrightarrow \J/\J^p K^\times \U_T^\Sigma\cdot$$
\end{lem}

\begin{proof}
Let us first describe $\alpha:  V_T^\Sigma/(K^\times)^p \rightarrow \prod_{\l \in \Sigma} K_\l^\times/(K_\l^\times)^pU_\l$.
 Take $x \in V_T^\Sigma$. Then $x\in \J^p \U_T^\Sigma$,  and $\alpha(x)$ is simply the projection to the $\Sigma$-coordinates. Thus, $ker(\alpha)= \J^p \U_T$ modulo $(K^\times)^p$ that is $V_T/(K^\times)^p$.

 The map $\beta : \prod_{\l \in \Sigma} K_\l^\times/(K_\l^\times)^p \rightarrow \J/ \J^pK^\times  \U_T$ is the inclusion followed by the restriction modulo $\J^pK^\times \U_T$.
 Obviously, $\beta \circ \alpha =0$, then $Im(\alpha) \subset ker(\beta)$. Let us study the reverse inclusion. Let $z:=(z_\l) \in \prod_{\l \in \Sigma} K_\l^\times $ be such that $z \in \J^pK^\times \U_T$ (in other words, $z\in ker(\beta)$).
 Then there exists $x\in K^\times$ such that $z=j^p \cdot x \cdot u$, where $j\in \J$ and $u\in \U_T$. Then $x \in K^\times \cap \J^p \U_T^\Sigma=V_T^\Sigma$, and $\alpha(x)=z$.

 The other maps are obvious.
\end{proof}

%\smallskip

Given a prime $\p$ of $K$, let $K_\p^{ur}$ be  the maximal unramified extension of $K_\p$; set $G_\p^{ur}:=Gal(K_\p^{ur}/K_\p)$.

\smallskip

The exact sequence of Lemma \ref{lem_longue_es} allows us to obtain the following proposition:

\begin{propo}\label{prop_longue_se}
 One has the exact sequence $$ H^1(G_T^\Sigma) \hookrightarrow H^1(G_T) \longrightarrow \bigoplus_{\l \in \Sigma} H^1(G_\l^{ur}) \longrightarrow M_{K,T}^\Sigma \twoheadrightarrow M_{K,T} ,$$
 where the map $H^1(G_\l^{ur}) \longrightarrow M_{K,T}^\Sigma$ relies on the Artin map and Kummer duality.
\end{propo}

\begin{proof}
 By the Artin maps, $\J/\J^p K^\times \U_T^\Sigma \simeq G^\Sigma_T/\Phi(G^\Sigma_T)$, and $K_\l^\times/(K_\l^\times)^pU_\l \simeq G_\l^{ur}/(G_\l^{ur})^p$.
 Then the  exact sequence of Lemma \ref{lem_longue_es} becomes
$$
  V_T/(K^\times)^p \hookrightarrow V_T^\Sigma/(K^\times)^p \longrightarrow \prod_{\l \in \Sigma} G_\l^{ur}/(G_\l^{ur})^p \longrightarrow G_T/\Phi(G_T)  \twoheadrightarrow G_\Sigma/\Phi(G_T^\Sigma).
 $$
 By taking the dual $\wedge$ we get
 $$ H^1(G_T^\Sigma) \hookrightarrow H^1(G_T) \longrightarrow \bigoplus_{\l \in \Sigma} H^1(G_\l^{ur}) \longrightarrow \left(V_{K,T}^\Sigma/(K^\times)^p\right)^\wedge \twoheadrightarrow  \left(V_{K,T}/(K^\times)^p\right)^\wedge \cdot$$
  To conclude, observe that by Kummer duality $$\left(V_{K,T}^\Sigma/(K^\times)^p\right)^\wedge \simeq M_{K,T}^\Sigma \ {\rm and,} \ \left(V_{K,T}/(K^\times)^p\right)^\wedge \simeq M_{K,T}.$$
\end{proof}

Let $\q$ be a tame prime of $K$.
By applying  Proposition \ref{prop_longue_se} successively with  $T=\emptyset$ and $T=\{\q\}$, we get the following commutative diagram:
 $$\xymatrix{ H^1(G_K)\ar@{->}[d] \ar@{->}[r]^-{\psi} & \displaystyle{\bigoplus_{\l \in \Sigma} H^1(G_\l^{ur})}\ar@{->}[r]^-{\varphi}\ar@{=}[d]& M_K^\Sigma  \ar@{->>}[d]^{\phi_\q}\ar@{->>}[r]& M_K\ar@{->>}[d] \\
 H^1(G_{K,T})  \ar@{->}[r] \ar@{->}[r]^-{\psi'}&\displaystyle{ \bigoplus_{ \l \in \Sigma}H^1(G_\l^{ur}) }\ar@{->}[r]^-{\varphi'} & M_{K,T}^\Sigma  \ar@{->>}[r]& M_{K,T}
 }
 $$

 Recall that $\ker(\phi_\q)$ is generated by the Frobenius element $\sigma_\q \in M_K^\Sigma$ (see Lemma~\ref{lemm_generate_frob}).

\smallskip

  Let us start with some local conditions $a:=(a_\l)_{\l \in S} \in \displaystyle{ \bigoplus_{ \l \in \Sigma}H^1(G_\l^{ur}) }$. Then $$a \in Im(\psi') \Longleftrightarrow a\in ker(\varphi')\Longleftrightarrow \varphi(a) \in ker(\phi_\q)\Longleftrightarrow  \langle \sigma_\q \rangle = \langle \varphi(a) \rangle\cdot $$

  \begin{lem}\label{lemm_local_conditions} Suppose  $a \notin Im(\psi)$. Let   $\q$ be a tame prime of $K$, not in $\Sigma$, such that in $M_K^\Sigma$, $\langle \sigma_\q \rangle = \langle \varphi(a) \rangle$. Then
  there exists a $\Z/p$-extension $N/K$ exactly ramified at~$\q$ that respects the $a_\l$'s,  $\l \in \Sigma$. Moreover, the tame prime $\q$ is such that $\sigma_\q  \in M_K^\Sigma$ restricts to  $M_K$ is trivial.
\end{lem}

\begin{proof}
 Let us choose a tame prime $\q \notin S$ such that  in $M_K^\Sigma$, $\langle \sigma_\q \rangle = \langle \varphi(a) \rangle$. Then there exists a $\Z/p$-extension $N/K$ unramified outside $\q$ that respects the $a_\l$, $\l \in \Sigma$. Since $a \notin Im(\psi)$, the extension $N/K$ is not unramified, and then $N/K$ is exactly ramified at~$\q$.  Moreover, the existence of such an extension implies that $\q$ splits totally in $M_K$ (see, for example, Theorem \ref{theorem_gras}, or observe that $\sigma_\q \in Im(\varphi)$).
\end{proof}

\begin{rem}
Observe now that a non-trivial $a_\l\in H^1(G_\l^{ur})$ indicates that $\l$ is inert in $N/K$.
\end{rem}

  We can prove the key proposition of our work.

\begin{propo} \label{prop_stab_inerte}
\label{prop_stab_inerte_bis} Suppose  $A_K \geq h_{G_K}^1$. Let $S=\{\l_1,\cdots \}$ be a finite set of primes of~$K$. Then there exists a $\Z/p$-extension  $N/K$  ramified at only one tame prime $\q$ such that:
\begin{itemize} \item[$(i)$] the extension $N/K$ is inert at all places of  $S$;
 \item[$(ii)$] there is stability of the $p$-tower, {\it i.e.},  $L_{p}(N) = NL_{p}(K)$.
\end{itemize}
\end{propo}

\begin{proof}
$\bullet$   For each $\l\in S$, take the non-trivial element $a_\l:=1 \in H^1(G_\l^{ur})$.

  Set $a'=(a_\l)_{\l \in S)} \in \bigoplus_{\l\in S} H^1(G_\l^{ur})$.

%  \smallskip

   $-$ If $a' \notin Im(\psi)$, set $a=a'$ and $\Sigma=S$.

 % \smallskip

   $-$ If $a' \in Im(\psi)$,   let us choose a prime $\l_0$ that splits totally in $(L_p^S)^{p,el}/K$, where $ (L_p^S)^{p,el}$ is the maximal elementary
abelian $p$-subextension
    of $L_p^S/K$.
   Set $\Sigma=S\cup\{\l_0\}$, and consider ${a=(1)_{\l \in \Sigma} \in \bigoplus_{\l\in \Sigma} H^1(G_\l^{ur})}$.
   By the choice of $\l_0$, $a\notin Im(\psi)$.

\medskip

$\bullet$ Let us consider the extensions of \S \ref{section_treillis} by replacing $S$ by $\Sigma$.

\smallskip

 Recall the isomorphism (\ref{iso_galois}): $$ Gal(F_5/F_{0}) = Gal(F_5/F_2)\times Gal(F_5/F_3),$$
 with $Gal(F_5/F_2)\simeq Gal(F_3/F_0)$.

%\smallskip

 Let  $z_0 \in Gal(F_5/F_2) $ such that its projection onto $Gal(F_3/F_0)$ coincides with $a$.

 \smallskip

 Let $d=h^1_{G_K}$, and let $z =  ((x_{1},...,x_{d}),0)\in M_{L_p(K)}$ as in Theorem  \ref{theorem_stabilite}. In fact, $z\in I_{G_K}(M_{L_p(K)})$, which indicates that $z \in  Gal(F_2/F_1)$ (see \S \ref{section_stabilite}). We then choose $z_1 \in Gal(F_5/F_4)$ such that its projection onto $Gal(F_2/F_1)$ coincides with $z$.

 \smallskip

 By the Chebotarev density theorem, we now choose a tame prime $\q$ of $K$ not belonging in  $\Sigma$, such that  $$\sigma_\q = (z_0,z_1) \in Gal(F_5/F_2)\times Gal(F_5/F_4) \subset Gal(F_5/K)\cdot$$

\medskip

$\bullet$
 Let's look at the implications of this choice.

% \smallskip

 First, by restriction  to $F_3$, $\sigma_\q=a \in Gal(F_3/F_0)  \subset M_K^\Sigma=Gal(F_3/K')$, which, by Lemma \ref{lemm_local_conditions} indicates the existence of a   $\Z/p$-extension $N/K$ ramified only at $\q$, such that every prime~$\l$ of $\Sigma$ is inert.

%\smallskip

Then, $\sigma_\q $ restricted to $F_2$ coincides with  $z$, which, by Theorem  \ref{theorem_stabilite}, implies the stability of the $p$-tower, {\it i.e.},   $L_p(N)=NL_p(K)$. Hence, the result.
\end{proof}

We now have all the elements to prove the main result of our work.

First, by assumption $\lambda_K\geq h^1_G$:  this is a consequence of \S \ref{section_minoration}.%Proposition \ref{MU2}.

By Proposition \ref{prop_stab_inerte}, there exists a  $\Z/p$-extension   $N/K$  ramified at some prime $\q$,   inert at each prime  $\p \in S$, and such that $L_p(N)=NL_p(K)$.

Let $\l \in S$ such that $\l$ is not totally splitting in $L_p(K)/K$. Recall that $\l$ is inert in $N/K$. Noting that the decomposition group of  $\l$ in $Gal(L_p(N)/N)$ is cyclic (defined up to conjugacy),  it is then a small exercise to observe that the residual degree of $\l$ in  $L_p(N)/N$ strictly decreases:  it is exactly divisible by $p$.
Indeed, set $\Gamma:=Gal(L_p(N)/K)$, $H=Gal(L_p(N)/L_p(K)\simeq \Z/p$ and $G':=Gal(L_p(N)/N) \simeq G$. Then $\Gamma=H\times G'$. Take a prime $\Ell|\l$ of $L_p(N)$, and consider the Frobenius $\sigma_\Ell$ of $\Ell$ in $\Gamma$; write $\sigma_\Ell=(h,g) \in H\times G'$. Since  $\l$ is inert in  $N/K$, then $h\neq 1$. Observe that the decomposition group of $\l$ in $Gal(L_p(K)/K)$ is generated by $g:=(1,g)$.
But $\sigma_\Ell^p=(1,g^p) \subset G'$; hence the decomposition group of $\Ell$ in $L_p(N)/N$ is generated by $(1,g^p)$ which is of order $p^{f-1}$, where $p^f$ is the order of $g$.
This is therefore true for any such prime  $\l \in S$.

If $\l \in S$ splits totally in  $L_p(K)/K$, then it also splits totally in $L_p(N)/N$.

 By iterating this process  $m$ times, we obtain a $p$-extension $F/K$  such that every prime   $\l\in S$ splits totally  in  $L_{p}(F)/F$. Thus, we have  $L^{S}_{p}(F) = L_{p}(F)$.

 To conclude, we need to estimate  $p^m$.  A coarse upper bound is  $\# G_K$.
 We can  go a bit further by bounding $p^m$ by the exponent of  $G_K$.

 \begin{rem}
 Observe also that each place $\l\in S$ is ``inert'' in the successive steps. In particular, in Theorem \ref{main_theorem},  $\# S_F=\# S$.
 \end{rem}

 \subsection{Proof of Corollary \ref{coroB}}

Let us conclude with a word on the proof of Corollary~\ref{coroB}.
Let  $G$ be a  $p$-group and let $K$ be a number field  such that  $Cl_K=\{e\}$. By the main theorem of  \cite{HMR}, there exists an  extension $\K/K$ such that  $Gal(L_p(\K)/\K) \simeq G$. Furthermore, the proof of this result shows that for the field $\K$ in question, we have   $A_{\K} \geq h^1_G$, which implies $\lambda_{\K} \geq h^1_G$. We can then apply Theorem \ref{main_theorem}.

%%%%%%%%%%%%

%%%%%%%%%%%%%%%%%%%%%%%%%%%%%%%%%%%%%%%%%
%%%%%%%%%%%%%%%%%%%%%%%

\section{Remarks}

\subsection{The condition $\H_S$} \label{section_remark_HS} Let  $S = \{\l_{1}, \l_{2},...,\l_{s}\}$ be a set of
primes of  $K$. We denote by $\H_S$ the following condition.

\medskip

$(\H_S)$:  Every prime  $\l\in S$ splits totally  in the elementary abelian extension   $(L_p)^{p,el}/K$ of $L_p(K)/K$.

\medskip

The condition $\H_S$ is therefore equivalent to the isomorphism between   $Cl_K/(Cl_K)^p $ and $ Cl_K^S/(Cl_K^S)^p$.

\smallskip

Observe that  the condition $\H_S$ is satisfied after a first application of Proposition \ref{prop_stab_inerte}. The goal here is to revisit the element $a$ and the choice of $\q$ in \S \ref{proof}.

\begin{lem} \label{lemmeHS}
 The  condition $\H_S$ is equivalent to  $\# Cl^{S}_K[p] = \# Cl_K[p]$.
 \end{lem}

 \begin{proof} This simply follows from the fact that for a finite abelian group~$A$ we have $\#A[p]=\# A/A^p$,   an equality derived from the exact sequence: $$1 \longrightarrow A[p] \longrightarrow A \stackrel{a\mapsto a^p}{\longrightarrow} A \longrightarrow A/A^p \longrightarrow 1.$$
\end{proof}

Lemma \ref{lemmeHS} allows us to prove the following lemma.

\begin{lem} \label{lemmeHSbis}
 Suppose $\H_S$. Then for every subset  $X \subset S$, we have the exact sequence: $$1 \longrightarrow V_{K}/(K^\times)^p  \longrightarrow   V_{K}^X/(K^\times)^p \longrightarrow (\Z/p)^{\# X} \longrightarrow 1 .$$
\end{lem}

\begin{proof} Let's start with the following commutative diagram:

 \begin{equation}
\xymatrix{
 1 \ar[r] & E_{K}/ E^{p}_{K} \ar@{^{(}->}[d]^{\alpha} \ar[r] & V_{K}/K^{\times p}\ar[r] \ar@{^{(}->}[d]^{\beta} & \ar[r] Cl_{K}[p]\ar[d]^{\gamma} & 1\\
1 \ar[r] & E_K^{X}/ (E_K^X)^p \ar[r] & V_{K}^X/K^{\times p}\ar[r] & \ar[r] Cl^{X}_{K}[p]& 1
}
\end{equation}
  Under the condition $\H_S$, by Lemma \ref{lemmeHS}, we known that   $ker(\gamma)$ and  $coker(\gamma)$ have the same order. This implies that  $coker(\alpha)$ and $coker(\beta)$ also have the same order by the Snake Lemma. Now, since  $coker(\alpha) \simeq (\Z/p)^{\# X}$ by Dirichlet's theorem, we obtain the result.\end{proof}

 \begin{lem}\label{lemm_compositum} Suppose   $\H_S$. Let $X=\{\l_{i_1},\cdots, \l_{i_x}\}\subset S$ be a subset of  $x$ primes of  $S$. Then $$K'(\sqrt[p]{V_{K}^X})= K'\left(\sqrt[p]{V_K^{\{\l_{i_1}\}}}, \sqrt[p]{V_K^{\{\l_{i_2}\}}}, \cdots ,\sqrt[p]{V_K^{\{\l_{i_x}\}}}\right),$$
and the Galois group of  $K'(\sqrt[p]{V_K^X})/K'(\sqrt[p]{V_K})$ is isomorphic to $(\Z/p)^{x}$.
 \end{lem}

 \begin{proof}
For $\# X=1$, this is the Lemma \ref{lemmeHSbis}.

%\smallskip

Suppose $X=\{\l_1,\l_2\} \subset S$. Obviously $V_K^{\{\l_1\}} V_K^{\{\l_2\}}  \subset V_{K}^X$, and $V_K^{\{\l_1\}} \cap V_K^{\{\l_2\}}({K'}^\times)^p=V_K $. By Lemma \ref{lemmeHSbis}, $Gal(K'(\sqrt[p]{V_K^{\{\l_i\}}})/K'(\sqrt[p]{V_K}))\simeq \Z/p$, and
$Gal(K'(\sqrt[p]{V_K^{X}}/K'(\sqrt[p]{V_K})\simeq (\Z/p)^2$, which proves the result.

\smallskip

Continue the process.
  \end{proof}

  We then obtain the following proposition.

\begin{propo}\label{prop_compositum} Suppose $\H_S$.  Then $$Gal\left(K'(\sqrt[p]{V_{K}^S})/K'(\sqrt[p]{V_K})\right) \simeq \prod_{i=1}^s Gal \left (K'(\sqrt[p]{V_K^{\{\l_i\}}})/K'(\sqrt[p]{V_K})\right) \simeq(\Z/p)^s\cdot$$
\end{propo}

\begin{proof} It is an immediate consequence of Lemma \ref{lemm_compositum}.
\end{proof}

We arrive at the following remark. Recall the isomorphism (\ref{iso_galois}): $ Gal(F_5/F_{0}) \simeq Gal(F_5/F_2)\times Gal(F_5/F_3)$ used in the proof of our main result (see \S \ref{proof}).
We have choosen a tame prime $\q$ of $K$ and an element $z_0 \in Gal(F_5/F_0)$ such that $\sigma_\q= z_0$ in $Gal(F_3/K')$.

Under $\H_S$, the element $z_0$ can be chosen as $(1,1,...,1) \in Gal(F_5/F_2) $ according to the isomorphism:
$$Gal(F_5/F_2)\simeq Gal(F_3/F_0) \simeq \prod_{i=1}^s Gal\left(K'(\sqrt[p]{V_K^{\{\l_i\}}})/K'(\sqrt[p]{V_K})\right) \simeq(\Z/p)^s\cdot$$

 By the Chebotarev density theorem,  choose a tame prime $\q$ of $K$ not belonging in  $S$, such that  $\sigma_\q = z_0 \in Gal(F_5/F_2)$.
 Since   $ \sigma_\q = 0$ in $M_K$, by Theorem \ref{theorem_gras}, there exists a    $\Z/p$-extension $N/K$ ramified only at $\q$.

Let $\l \in S$. The restriction of $\sigma_\q$ to $Gal\left(K'(\sqrt[p]{V_K^{\{\l\}}})/K'(\sqrt[p]{V_K})\right)$ is non-trivial by the choice of $z_0$, which, by Corollary\ref{coro_gras}, implies that   $\l$ is inert in $N/K$.

%%%%%%%%%%%%%%%%%%%%%%%%%

\subsection{On the degree}

Let $G$ be a pro-$p$ group. Let $d=h_G^1$ and $r=h_G^2$. Recall the Golod-Shafarevich criterion   (see \cite{GS}, \cite[\S 7.7]{Koch} or \cite{R}):  If $G$ is finite, then  $r> d^2/4$.

\smallskip

On the other hand, when $G=Gal(L_p^S(K)/K)$, according to Shafarevich and Koch, we know that $0\leq r -d \leq r_{K,1}+r_{K,2}+\# S$ (see for example \cite[Chapter X, Theorem 10.7.12]{NSW}).
Therefore, if  $L_p^S(K)/K$ is finite then $r_{K,1}+r_{K,2}+ \# S> d^2/4-d$.
Thus, when $\#S $ is bounded, the degree $[K:\Q]$ grows according to the $p$-rank of $G$.

\subsection{On  $T$-ramified and $S$-split $p$-Hilbert ray class field towers}
Let $p$  be a prime number. Let  $F$ be a number field, and
let $S$ and $T$ be two finite and disjoint sets of primes of $F$. We assume that the primes in $T$ are tame.
Let $L_{p,T}^S(F)$ be the pro-$p$ extension of $F$, unramified outside  $T$,  totally splitting at $S$, and maximal for these properties.
Set $G_{F,T}^S:=Gal(F_T^S/F)$.

Recall that if $\#T$ is large compared to the degree of  $F/\Q$ and  $\# S$ is fixed, then $G_{F,T}^S$ is infinite. This can be seen, for example, through genus theory  (see the main theorem of \cite{M0}) associated with the  Golod-Shafarevich theorem. See also, for example, \cite{M}.
Here, we make the following observation.

 \begin{thm}  Let $K$ be a number field such that $Cl_K=1$, and let $S$ be a finite set of primes of $K$. Let $G$ be  a $p$-group and let $n\geq 0$.
 Then there exists an extension $F/K$ and a set  $T$ of tame primes of  $F$ such that:
  \begin{itemize}
   \item[$(i)$] $\# T  = n$,
   \item[$(ii)$] $G_{F,T}^S\simeq G$.
  \end{itemize}
 \end{thm}

 As before, we will abusively denote by $S:=S_F$ the set of primes of $F$ above those in $S$.

\begin{proof}
We begin by applying Corollary \ref{coroB}: there exists an extension  $\FF/K$ such that  $L_p^S(\FF)=L_p(\FF)$ and  $Gal(L_p(\FF)/\FF)\simeq G$. We then note that the proof guarantees enough Minkowski units, {\it i.e.}, $\lambda_{\FF} \geq h_G^1$. By %Propositions \ref{MU1} and
\S \ref{section_minoration}, %\ref{MU2},
it  is then possible to use the stability theorem~\ref{theorem_stabilite} to obtain an extension  $F/\FF$ such that $L_p^S(F)=L_p(F)$, $Gal(L_p(F)/F)\simeq G$, and $\lambda_F \geq n$.

Thus $M_{L_p(F)}= \F_p[G]^n \oplus M_0$. For $i=1,\cdots, n$, define $x_i=((0,\cdots,0, 1,0,\cdots 0), 0) \in M_{L_p(F)}$. As an   $\F_p[G]$-module, the elements $x_i$ form a basis of a free subspace of dimension~$n$. For  $i=1,\cdots, n$, by the Chebotarev density theorem, choose a tame prime~$\q_i$ of~$F$ such that  $\sigma_{\q_i}=x_i \in M_{L_p(F)}$.

By the theorem of Gras \ref{theorem_gras}, there is no extension of  $L_p(F)$ that is exactly ramified at any non-trivial subfamily of   $T=\{\q_{1},\cdots, \q_n\}$  (see also Remark \ref{rem_artin}).
Since  $p\nmid \# Cl_{L_p(F)}$, we conclude that $L_p(F)=L_{p,T}^S(F)$. Hence, the result follows.
\end{proof}

%%%%%%%%%%%%%%%%%%%%%%%%%%%%%%%%%

%%%%%%%%%%%%%%%%%%%%%%%%%%%%ù
%%%%%%%%%%%%%%%%%%%%%%%%%%%%%%

\end{document}